\documentclass{amsart}
\usepackage[utf8]{inputenc}
\usepackage{amsmath}
\usepackage{amssymb}
\usepackage{amsthm}
\usepackage[style=alphabetic]{biblatex}
\usepackage{cleveref}
\usepackage[a4paper]{geometry}

\newtheorem{theorem}{Theorem}[section]
\newtheorem{proposition}{Proposition}[section]
\newtheorem{lemma}{Lemma}[section]
\newtheorem{defn}{Definition}[section]
\newtheorem{corollary}{Corollary}[section]

\newtheorem{example}{Example}

\DeclareMathOperator{\Sym}{Sym}

\title{Cyclic Cubic Points on Higher Genus Curves}
\author{James Rawson}
\address{James Rawson, Mathematics Institute, Zeeman Building, University of Warwick, Coventry, CV4 7AL}
\email{james.rawson@warwick.ac.uk}
\thanks{The author is supported by the Warwick Mathematics Institute Centre for Doctoral Training, and gratefully acknowledges funding from the UK Engineering and Physical Sciences Research Council (Grant number: EP/W523793/1)}

\begin{document}

\begin{abstract}
 The distribution of degree $d$ points on curves is well understood, especially for low degrees. We refine this study to include information on the Galois group in the simplest interesting case: $d = 3$. For curves of genus at least 5, we show cubic points with Galois group $C_3$ arise from well-structured morphisms, along with providing computable tests for the existence of such morphisms. We prove the same for curves of lower genus under some geometric or arithmetic assumptions.
\end{abstract}

\maketitle

\section{Introduction}
Let $X$ be a smooth, projective, geometrically integral curve of genus $g \geq 2$ defined over a number field $K$. A cubic point of $X$ is an element of $X(L) \setminus X(K)$ for $L$ a degree 3 extension of $K$, and the Galois orbit of such a point gives rise to a $K$-point on $\Sym^3 X$. Conversely, a $K$-point of $\Sym^3 X$ is either the orbit of a cubic point, a rational point plus the orbit of a quadratic point, or the sum of three rational points. Using the natural map to the Jacobian of $X$, Abramovich and Harris show that $\Sym^3 X(K)$ is infinite if and only if $X$ has a morphism of degree at most 3 to either a rational curve or an elliptic curve with positive rank \cite{AbrHar}. 

Similarly, a cyclic cubic point can be defined as an element of $X(L) \setminus X(K)$ where $L$ is a degree 3 extension of $K$ and the Galois group of $L / K$ is cyclic. There is also a correspondence between orbits of such points and $K$-points of the quotient variety $X^3 / C_3$ - the Galois orbit of a cyclic cubic point gives a $K$-point of $X^3 / C_3$, and a $K$-point of $X^3 / C_3$ is either the orbit of a cyclic cubic point or three rational points. By analogy with the result of Abramovich and Harris, it might be expected that there are infinitely many points on $X^3 / C_3$ if and only if $X$ has a degree 3 map to a rational curve or elliptic curve with Galois group $C_3$. This is false, as shown by this example.

\begin{example}
The curve, given by the following affine equation, has infinitely many cyclic cubic points without admitting a morphism to $\mathbb{P}^1$ or an elliptic curve with Galois group $C_3$: $$y^3 - 4(27 x^{10} + x^3 - 16x + 16)y = 16x^5 (27 x^{10} + x^3 - 16x + 16).$$
As it has genus $10$, the Castelnuovo--Severi inequality (see \Cref{csineq}) shows it can only triple cover either $\mathbb{P}^1$ or an elliptic curve, as such curves have genus at most 7. Projecting onto the $x$-coordinate gives a map to $\mathbb{P}^1$, and by Castelnuovo--Severi, this is again unique. The polynomial $y$ satisfies over $\mathbb{Q}(x)$ has discriminant $256(27x^{10} + x^3 - 16x + 16)^2(x^3 - 16x + 16)$, which is not a square, and so the extension is not Galois.  However, this is a square for any rational $x$ such that $x^3 - 16x + 16$ is a perfect square, and so the points coming from these fibres of the map are cyclic cubic. The condition $x^3 - 16x + 16 = \Delta^2$ is the elliptic curve with Cremona label 37a1, which has rank 1, and so the curve has infinitely many cyclic cubic points. 
\end{example}

For curves of genus at least 5, we show that this kind of construction is the only source of infinite families of cyclic cubic points, generalising the work in \cite{DerNaj}.

\begin{theorem}
Let $X / K$ be a curve of genus $g \geq 5$, then $X$ has infinitely many cyclic cubic points if and only if there exists a degree 3 morphism $f : X \to Y$ such that $Y$ is a $\mathbb{P}^1$ or a positive rank elliptic curve, and the ``discriminant curve'' of $f$ (see \Cref{discrim}) is also a $\mathbb{P}^1$ or a positive rank elliptic curve.
\label{mainthm1}
\end{theorem}

We also provide computable conditions on $X$ that show no morphisms to $\mathbb{P}^1$ or an elliptic curve with this property exist.

\begin{theorem}
Let $X$ be a curve of genus $g \geq 5$. Assume the following
\begin{itemize}
 \item $X$ has no $K$-rational automorphism of order 3
 \item No $K$-rational, geometrically integral, unramified double cover of $X$ has such an automorphism
 \item For any point $P \in X(\bar{K})$, the number of points $Q \in X(\bar{K})$ such that $3P - 3Q = \mathrm{div}(f)$ is strictly less than $g$.
\end{itemize}
Then $X$ has only finitely many points defined over cyclic cubic extensions of $K$.
\label{mainthm_comp}
\end{theorem}

For lower genus curves, it is possible to get similar results, but at the cost of much stronger geometric constraints.

\begin{theorem}
Let $X$ be a curve of genus $g \geq 3$. Assume the following
\begin{itemize}
 \item For $g = 3$, the Bombieri--Lang conjecture holds for $X^3 / C_3$
 \item The Jacobian of $X$ contains no elliptic curves
 \item For any point $P \in X(\bar{K})$, the number of points $Q \in X(\bar{K})$ such that $3P - 3Q = \mathrm{div}(f)$ is strictly less than $g$.
\end{itemize}
Then $X$ has only finitely many points defined over cyclic cubic extensions of $K$.
\label{mainthm2}
\end{theorem}

Using arithmetic data, the genus can be lowered even further to 2.
\begin{theorem}
Let $X$ be a curve of genus $g \geq 2$. Assume the following
\begin{itemize}
 \item The rank of the Jacobian of $X$ is 0
 \item If $g \geq 3$, for any point $P \in X(\bar{K})$, the number of points $Q \in X(\bar{K})$ such that $3P - 3Q = \mathrm{div}(f)$ is strictly less than $g$.
 \item If $g = 2$, $J_X$ has no 3-isogenies and $X$ does not have an automorphism of order 3.
\end{itemize}
Then $X$ has only finitely many points defined over cyclic cubic extensions of $K$.
\label{mainthm3}
\end{theorem}

Alternatively, puncturing the curve and considering integral points also allows for small genus results.
\begin{theorem}
Let $X$ be a curve of genus $g \geq 2$, and $\mathcal{X}$ be an integral model of a proper open subset of $X$. Assume the following
\begin{itemize}
 \item For any point $P \in X(\bar{K})$, the number of points $Q \in X(\bar{K})$ such that $3P - 3Q = \mathrm{div}(f)$ is strictly less than $g + 1$.
 \item If $g = 2$, assume that complement of (the generic fibre of) $\mathcal{X}$ in $X$ contains at least 4 $K$-rational points.
\end{itemize}
Then $X$ has only finitely many integral points defined over cyclic cubic extensions of $K$.
\label{mainthm4}
\end{theorem}

\subsection{Outline}
This paper is organised into 5 sections. The first of these concerns morphisms and cyclic cubic points, especially morphisms to $\mathbb{P}^1$. The second gives a proof of a result from \cite{AbrHar} as this paper is known to have some gaps (\cite{DebFah, KadVog}). The third combines the first two to prove Theorems \ref{mainthm1} and \ref{mainthm_comp}. The fourth is devoted to understanding what happens when the genus constraint is dropped and proves Theorems \ref{mainthm2} and \ref{mainthm3}. The last concerns integral points, and proves \Cref{mainthm4}.

\subsection{Conventions}
Throughout this paper, $X$ and $Y$ will be smooth, projective, geometrically integral curves over a number field $K$ and $X$ will have genus at least 2.
We will also assume the existence of a degree 3 divisor on $X$ to define the Abel--Jacobi map from $\Sym^3 X$ to $J_X$. If there is no such divisor, then no cubic points exist, and so no cyclic cubic points exist either.

\subsection*{Acknowledgements}
The author would like to thank Samir Siksek for many productive conversations surrounding this topic, as well as his support and supervision. He also thanks Borys Kadets for interesting conversations and the anonymous referee whose detailed comments have improved this paper.

\section{Maps to $\mathbb{P}^1$}
In this section, we establish some definitions and basic theorems that will be important through the rest of the paper about morphisms between curves, especially to $\mathbb{P}^1$.

The following shorthand will be useful for many statements.
\begin{defn}
Let $f : X \to Y$ be a degree 3 morphism. The \emph{discriminant curve} of $f$, $Y_{\Delta(f)}$, is the smooth curve corresponding to the function field $K(Y)(\sqrt{\Delta})$ where $\Delta$ is the discriminant of the field extension $K(X) / K(Y)$ coming from $f$. This curve is irreducible over $K$, but may not be geometrically irreducible (for example, if $\Delta$ is a non-square constant).
\label{discrim}
\end{defn}

A condition on the discriminant curve is enough to guarantee the existence of cyclic cubic points, as follows. 

\begin{theorem}
Suppose $f : X \to Y$ is a degree 3 morphism, with $Y$ either a $\mathbb{P}^1$ or a positive rank elliptic curve. Then $f$ has infinitely many fibres which are cyclic cubic if and only if $Y_{\Delta(f)}$ is either a $\mathbb{P}^1$ or a positive rank elliptic curve.
\label{mapinf}
\end{theorem}
\begin{proof}
First, assume $Y_{\Delta(f)}$ is a $\mathbb{P}^1$ or a positive rank elliptic curve. Let $S$ be the image of $Y_{\Delta(f)}(K)$ in $Y(K)$, which is an infinite set by assumption. For a point $s \in S$, either $f^{-1}(s)$ contains a $K$-rational point, or it is irreducible. As $X$ has genus at least 2, it has finitely many $K$-rational points, and so the former can only occur for finitely many $s$.

Let $P \in X(\bar{K})$ be in the fibre above $s$. Let $x$ be a coordinate of $X$ realising the extension of function fields induced by the morphism, then $K(P) = K(x(P))$. The discriminant of the field extension is, up to squares in $K(Y)$, the discriminant of the minimal polynomial of $x$. This minimal polynomial specialises to the minimal polynomial of $x(P)$ at $s$. The discriminant of the minimal polynomial of $x$ is a square when specialised at $s$, and so the discriminant of the minimal polynomial of $x(P)$ is also a square in $K(Y)$ and hence $K$. This shows $K(x(P))$ is a cyclic extension of $K$. 

Now instead, assume $f$ has infinitely many cyclic cubic fibres. For infinitely many $s \in Y(K)$, there exists a $P \in f^{-1}(s)(\bar{K})$ such that $K(P)$ is a cyclic cubic extension of $K$. As before, take a coordinate, $x$, on $X$ which generates the function field extension given by the morphism. The discriminant of the minimal polynomial of $x(P)$ is a square as $K(P)$ is cyclic and generated by $x(P)$. This implies the discriminant of the minimal polynomial of $x$ is a square when specialised to $s$, and so the discriminant of the field extension is a square after specialising. In particular, $s$ is in the image of the rational points of $Y_{\Delta(f)}$. As this happens for infinitely many $s$, $Y_{\Delta(f)}(K)$ is infinite and $Y_{\Delta(f)}$ is either a $\mathbb{P}^1$ or a positive rank elliptic curve. 
\end{proof}

We can also give a weakening to the backward implication: that under certain conditions, having infinitely many cyclic cubic points implies the existence of a morphism as in the previous theorem. To state this, we first make a definition.

\begin{defn}
 Fix a degree 3 divisor, $D_0$. Then $W^{(3)}_X$ is defined to be the image of the Abel--Jacobi map (with basepoint $D_0$) from $\Sym^3 X$ to $J(X)$. This is well-defined up to translation, as a different choice of $D_0$ results in a translation by the difference of the two basepoints.
\end{defn}

We can now state this weakening.
\begin{theorem}
If $(X^3 / C_3)(K) \to W^{(3)}_X(K)$ is not finite-to-one, there exists a map $f : X \to \mathbb{P}^1$ where $\mathbb{P}^1_{\Delta(f)}$ is either a $\mathbb{P}^1$ or a positive rank elliptic curve.
\label{finmap}
\end{theorem}
\begin{proof}
As $(X^3 / C_3)(K) \to \Sym^3 X(K)$ is finite-to-one, the failing of $(X^3 / C_3)(K) \to W^{(3)}_X(K)$ to be finite-to-one implies $\Sym^3 X(K) \to W^{(3)}_X(K)$ is not finite-to-one. Let $D$ be a point of $\Sym^3 X(K)$ such that the fibre of $\Sym^3 X \to W^{(3)}_X$ containing $D$ has infinitely many rational points. As the fibre has infinitely many rational points, it is positive dimensional, and so $D$ moves as a divisor, and the points of the fibre are linearly equivalent divisors to $D$. Moreover, as the genus of $X$ is at least 2, the linear system attached to $D$ is 1-dimensional (if it was 2-dimensional, this would give a base-point-free linear system of degree 3 on $X$, making it birational to a plane cubic). If the linear system gives a degree 2 morphism to $\mathbb{P}^1$, then the family of divisors has a fixed point, which cannot happen for a family of rational points on $X^3 / C_3$ (if it did, the fixed point would need to be rational, and then the other 2 points should be rational, giving rise to infinitely many rational points on $X$). We can therefore assume that $X$ has a degree 3 map to $\mathbb{P}^1$, and the fibres of this map are the divisors arising in the contracted $\mathbb{P}^1$. By the previous theorem, this shows that the discriminant curve of the morphism is either a $\mathbb{P}^1$ or a positive rank elliptic curve.
\end{proof}

We now give criteria to determine whether a curve has such a morphism to $\mathbb{P}^1$.
\begin{proposition}
Suppose that there exists a morphism $f : X \to \mathbb{P}^1$ of degree 3 such that the discriminant curve is either genus 0 or 1, then there exists distinct points $P_1, ..., P_g \in X(\bar{K})$ such that $3P_i \sim 3P_j$ for all $i, j$.
\label{weier}
\end{proposition}
\begin{proof}
As $f$ is a degree 3 morphism from $X$ to $\mathbb{P}^1$, it has $2g + 4$ ramification points (counting with multiplicity) over the algebraic closure of $K$. The discriminant vanishes to order one at the simple branch points, and so the discriminant curve is ramified over this point. To meet the genus requirement on the discriminant curve, there can be at most 4 ramification points of $f$ that have index 2. This leaves at least $\frac{2g}{2} = g$ ramification points, $P_1, ..., P_g \in X(\bar{K})$ of $f$ with index 3. As $f$ is degree 3 and each $P_i$ has ramification index 3, $3P_i$ is a fibre of $f$ for each $i$. The linear equivalence statement follows, using suitable rational functions of $f$.
\end{proof}
For curves of genus at least 3, for a point to satisfy this linear equivalence it must be a Weierstrass point. 

For curves of genus 2, this proposition is insufficient, as was pointed out to the author by Borys Kadets, since any 3-torsion point can be written as $P - Q$, for some points $P, Q$ on the curve and $3P \sim 3Q$. In particular, this shows that every genus 2 curve, over a suitable field extension, has infinitely many cyclic cubic points. In this case, the following refinement is necessary.
\begin{proposition}
Let $X$ be a genus 2 curve, and suppose there exists a degree 3 morphism $f : X \to \mathbb{P}^1$, such that the discriminant curve is a $\mathbb{P}^1$ or an elliptic curve. Then either $X$ has an automorphism of order 3, or the Jacobian of $X$ has a 3-isogeny. 
\label{gen2weier}
\end{proposition}
\begin{proof}
As $f$ is a degree 3 morphism $X$ to $\mathbb{P}^1$, it has 8 ramification points over $\bar{K}$, counting multiplicity. For the discriminant curve to have the necessary genus, at least 2 of these ramification points have ramification index 3.

If it has exactly 2, $P$ and $Q$, then $3P \sim 3Q$ so $P - Q$ is a 3-torsion point. As $P$ and $Q$ are the only 2 ramification points with index 3 for $f$, they are either rational, or conjugate to each other. In either case, the subgroup of the Jacobian generated by $P - Q$ is defined over $K$, giving rise to a 3-isogeny. 

If there are exactly 3 ramification points with index 3, say $P_1, P_2$ and $P_3$, then $P_1 - P_3$ and $P_2 - P_3$ are 3-torsion points, and so is $P_1 + P_2 - 2P_3$. As the $P_i$ are closed under Galois conjugation, this point is rational as any permutation of $1, 2, 3$ fixes this. For example, swapping 2 and 3 gives the divisor $P_1 + P_3 - 2P_2 = (P_1 - P_3) + (2P_3 - 2P_2)$, but $2P_3 - 2P_2 \sim P_2 - P_3$ as $P_3 - P_2$ is a 3-torsion point. This shows $P_1 + P_3 - 2P_2 \sim P_1 + P_2 - 2P_3$. The subgroup generated by this point is $K$-rational, and so gives a 3-isogeny.

Finally, if there are 4 such ramification points, this exhausts the ramification. The discriminant curve is then the trivial extension and the discriminant of $f$ is a perfect square, and so $f$ has Galois group $C_3$, which shows $X$ has an automorphism of order 3.
\end{proof}
Since the 3-torsion of genus 2 curves can be effectively computed \cite{DokDor}, as can their automorphism groups, these conditions can be verified.

We close the section with a classical result of a somewhat different flavour. This will be used to exclude the existence of certain morphisms on the curve.
\begin{theorem}[Castelnuovo--Severi]
 Let $X$ be a curve. Suppose $X$ has maps $f_i : X \to Y_i$, $i = 1, 2$ of degree $d_i$ and the genus of $Y_i$ is $g_i$. Then either the genus of $X$ is at most $d_1 g_1 + d_2 g_2 + (d_1 - 1)(d_2 - 1)$, or the $f_i$ factor through a common non-trivial map. 
 \label{csineq}
\end{theorem}

\section{A Theorem of Abramovich and Harris}
The previous section describes the (cyclic) cubic points of $X$ that arise from maps to $\mathbb{P}^1$. The remainder are controlled by rational points on $W^{(3)}_X$. To help understand this, we have the following theorem of Abramovich and Harris \cite{AbrHar}. 

\begin{theorem}
If $g(X) \geq 4$, and $W^{(3)}_X$ contains an abelian variety, $A$, then the abelian variety is an elliptic curve. If, further, $g(X) \geq 5$, then $X$ admits a map to $A$ of degree at most 3. Moreover, the fibres of this map are the divisors appearing in $A$, less fixed points.
\label{abrhar}
\end{theorem}

The paper this result comes from is known to have some results for which the proof is incomplete (see \cite{DebFah} and \cite{KadVog} for further details). In \cite{KadVog}, they prove the existence of a map to an elliptic curve under the assumptions of this theorem, but they do not show the second half, and their work assumes that the gonality of $X$ is more than 3. We therefore give a proof here, based on the methods of \cite{KadVog}.

Curves of genus at least 4 have at most finitely many morphisms of degree at most 3 to $\mathbb{P}^1$ (up to automorphisms of $\mathbb{P}^1$), and so the Abel--Jacobi map is birational onto its image. In particular, $A$ may be replaced by its strict transform inside $\Sym^3 X$. If every element of $A(K) \subset \Sym^3 X(K)$ contains a fixed point, then we can embed $A$ in $\Sym^2_X$ instead.

We first recall a few lemmas, starting with \cite[Lemma 1]{AbrHar} (or \cite[Proposition 3.3]{KadVog}).
\begin{lemma}
 If $D \in A(K) \subset \Sym^d_X(K)$, for some $d$, then $\dim |2D| \geq \dim A$.
\end{lemma}

When $A$ is an elliptic curve, we can improve this estimate as in \cite[Lemma 2]{AbrHar} and \cite[Lemma 2.2 and Proposition 3.3]{KadVog}.
\begin{proposition}
 Let $D$ be in $A(K) \subset \Sym^d_X(K)$ for some $d$. If $\dim A = 1$ and $\dim |2D| = 1$, then there exists a morphism of degree $d$ from $X$ to $A$, such that each point of $A(K)$ is contained in a fibre. If, instead, a general point of $X(\bar{K})$ appears exactly once in the support of elements of $A(\bar{K})$, the same result is true.
 \label{fibre_lemma}
\end{proposition}

On the other hand, there is an upper bound on the size of $\dim |2D|$ as follows.
\begin{lemma}
 For a general $D \in A(\bar{K}) \subset \Sym^d_X(\bar{K})$, $\dim |2D| \leq d - 1$ or the genus of $X$ is at most $d$.
\end{lemma}
\begin{proof}
 Let $D$ be general and such that $\dim |2D| \geq d$. Then as $\deg 2D = 2d$ and $h^0(X, 2D) \geq d + 1$, either this $2D$ is not special ($h^0(X, K_X - 2D) = 0$) or they are the canonical divisor or a multiple of the hyperelliptic divisor (if one exists) by Clifford's Theorem. By genericity, such equivalences cannot hold and so $2D$ is not special. By the Riemann--Roch Theorem, this forces the genus of the curve to be at most $d$.
\end{proof}

If $A$ is contained in $\Sym^2_X$, this completes the proof of \Cref{abrhar} as $\dim |2D| = 1$ for a general $D \in A(\bar{K})$.

This leaves the case where $A$ is not contained in $\Sym^2 X$ and $\dim |2D| = 2$ for a general $D \in A(\bar{K})$. The next lemma, based on \cite[Theorem 3.5]{KadVog}, controls the behaviour of this linear system.
\begin{lemma}
 The linear system $|2D|$ is birational for a general $D \in A(\bar{K}) \subset \Sym^3 X(\bar{K})$ if $\dim |2D| = 2$.
\end{lemma}
\begin{proof}
 Suppose not, then the induced map $\phi : X \to \mathbb{P}^2$ factors as $f : X \to Y$ and $Y \hookrightarrow \mathbb{P}^2$. Let $m$ denote the degree of $f$, then $6 = \deg 2D = m \deg Y$, and since $Y$ is non-degenerate (it is the image of a complete linear system), $\deg Y \geq 2$. This shows $m = 2$ or $3$ and $\deg Y = 3$ or $2$ respectively, so the genus of $Y$ is at most 1. As already noted, $X$ has only finitely many maps to curves of genus 0 of degree at most 3. Similarly, by a theorem of Kani \cite[Corollary after Theorem 4]{Kani}, it has at most finitely many maps to genus 1 curves of degree at most 3, up to automorphisms of the target. As there are finitely many maps, $D$ can be assumed to be disjoint from the branch locus of these maps, and varying $D$ shows that at least one of these maps occurs infinitely often as $X \to Y$. As noted in \cite[Theorem 3.5]{KadVog}, we have $D = f^{-1}(f(D))$ as sets, and so $3 = \#D = \# f^{-1}(f(D)) = m \# f(D)$. This rules out $m = 2$, and if $m = 3$, then $f(D)$ is a point. This shows a general $D$ has $h^0(X, D) \geq 2$, contradicting the fact that $X$ has only finitely many maps to $\mathbb{P}^1$ of degree at most 3.
\end{proof}

It remains to exclude the possibility that $|2D|$ is birational onto its image. We start with the following, based on \cite[Lemma 4.5]{KadVog}.
\begin{lemma}
 Assume $\dim |2D| = 2$ and is birational onto its image for a general $D \in A(\bar{K})$. Let $x$ be in the support of $D$, and $D'$ another element of $A(\bar{K})$ containing $x$, then $D' \cap D = \{x\}$.
 \label{good_intersection}
\end{lemma}
\begin{proof}
 See the proof of \cite[Lemma 4.5]{KadVog}.
\end{proof}

We handle the two cases in the theorem separately, starting with the genus of $X$ being at least 5, following \cite[Proposition 4.7]{KadVog}.

\begin{proposition}
 If the genus of $X$ is at least 5, and a general $D \in A(\bar{K}) \subset \Sym^3 X(\bar{K})$ has $\dim |2D| = 2$, then the second half of \Cref{fibre_lemma} is met.
\end{proposition}
\begin{proof}
 We start by bounding $\dim |3D|$. Write $3D \sim D_1 + D_2 + D_3$, where $D_i \in A(\bar{K})$ and the $D_i$ are disjoint, then $D_1 + D_2 \sim 2E$ and $D_1 + D_3 \sim 2E'$ for some $E, E' \in A(\bar{K})$ as abelian varieties are divisible groups. In particular, $h^0(X, D_1 + D_2), h^0(X, D_1 + D_2) \geq 3$. The intersection of these two spaces of sections is $H^0(X, D_1)$, but for a general choice of $D_1$, this is spanned by the constant functions. This shows $h^0(X, D_1 + D_2 + D_3) \geq 5$, so $\dim |3D| \geq 4$. 
 On the other hand, if $\dim |3D| \geq 5$, then $3D$ is a degree 9 divisor with $h^0(X, 3D) \geq 6$, which by Clifford's Theorem is not special. This forces the genus of $X$ to be at most 4.
 
 It remains to exclude the case where $\dim |3D| = 4$. For a general $D' \in A(\bar{K})$, $\dim |3D - D'| = 2$ (as this is equivalent to $2E$ for some $E \in A(\bar{K})$). This linear system is equivalent to the one obtained from $|3D|$ by projecting away from the span of $D'$, which shows the span of $D'$ is a line. Either the assumption of the second part of \Cref{fibre_lemma} holds, or for $x \in D'$, there exists a $D'' \in A(\bar{K})$ with $D' \neq D''$ such that $x \in D' \cap D''$. By the previous lemma, $\{x\} = D' \cap D''$ as $D'$ general. The linear system $|2D + D'|$ is 4-dimensional and $D''$ spans a line. Projecting away from the span of $D'$ contracts $D''$ to a point, as they meet at $x$. However, this is the same as the image of $D''$ under $|2D|$, and as the image of $D''$ under $|2D|$ is generally a line (as $D$ and $D''$ are a general pair). This is a contradiction, so the system is not birational onto its image.
\end{proof}

To rule out abelian surfaces (higher dimensional abelian varieties are already excluded by the upper bounds on the dimension of $|2D|$), we note that the assumptions of \Cref{fibre_lemma} cannot be met, and so $\dim |2D| = 2$ and is birational for a general $D \in A(\bar{K})$. There is a map from $A$ to $\mathbb{P}^2 = |2D|$ given by $D' \mapsto D' + D''$, where $D''$ is the element of $A$ equivalent to $2D - D'$. This map factors through the Kummer surface of $A$ (given by the involution $D' \mapsto D''$), and there are no generically injective maps from a Kummer surface to $\mathbb{P}^2$. This implies that for a general $D'$, there exists a distinct pair $E, E' \in A(\bar{K})$ such that $D' + D'' = E + E'$ as divisors. As all 4 divisors have degree 3, we may assume that after swapping $E$ and $E'$, $\# (D' \cap E) = 2$, which contradicts \Cref{good_intersection}. 

\section{Genus $\geq 5$ curves}
We now apply the results of the previous section to the study of cyclic cubic points on curves of genus at least 5 to prove \Cref{mainthm1}.
\begin{theorem}
Let $X$ be a curve of genus at least 5, then $X$ has infinitely many cyclic cubic points if and only if there exists $f : X \to Y$ where $\mathrm{deg}(f) = 3$, and $Y$ is either a $\mathbb{P}^1$ or a positive rank elliptic curve, and the same holds for $Y_{\Delta(f)}$. Moreover, all but finitely many cyclic cubic points on $X$ come from morphisms of this type.
\label{classgen5}
\end{theorem}
\begin{proof}
The curve $X$ has infinitely many cyclic cubic points if and only if $X^3 / C_3$ has infinitely many rational points, so we assume $X^3 / C_3$ has infinitely many rational points. There is a sequence of maps $X^3 / C_3 \to \Sym^3 X \to W^{(3)}_X$ given by taking the further quotient, and then applying the Abel--Jacobi map.

We first assume the image of the set of rational points on $X^3 / C_3$ is not a finite subset of $W^{(3)}_X$. By Faltings' Theorem \cite{Fal}, in $W^{(3)}_X$ all but finitely many rational points come from subvarieties that are positive rank abelian varieties. As there are infinitely many rational points on $W^{(3)}_X$ from $X^3 / C_3$, there must be (a translate of) an abelian subvariety that contains infinitely many of these points. By \Cref{abrhar}, this abelian variety is a positive rank elliptic curve, $E$, and the points of $E$, viewed as divisors, are the fibres of a map $X \to E$. If the map is not 3-1, the family of divisors has a basepoint, this would force $X$ to have infinitely many rational points, as in Theorem~\ref{finmap}. The map can therefore be assumed to be 3-to-1. As infinitely many of the fibres are defined over cyclic cubic fields, by \Cref{mapinf}, the discriminant curve is also a positive rank elliptic curve.

Next, assume instead that the rational points of $X^3 / C_3$ map to only finitely many points in $W^{(3)}_X$. This shows the map $(X^3 / C_3)(K) \to W^{(3)}(K)$ is not finite-to-one, and so by Theorem~\ref{finmap}, $X$ has a degree 3 map to $\mathbb{P}^1$ and the discriminant curve is either a $\mathbb{P}^1$ or a positive rank elliptic curve.  

For the ``all but finitely many'' statement, the above steps can be repeated with the rational points of $X^3 / C_3$ with the points associated to a known morphism removed.

The converse is Theorem~\ref{mapinf}.
\end{proof}

This theorem can be impractical, since it requires understanding all maps of degree 3 from the curve to genus 0 and 1 curves. To simplify this, we show that these conditions can be replaced by much more directly computable constraints. The Castelnuovo--Severi inequality (\Cref{csineq}) shows any curve with 2 degree 3 maps to $\mathbb{P}^1$ has genus at most 4, so genus 5 curves have at most 1. Combining this with the condition imposed on Weierstrass points given by Proposition~\ref{weier}, simplifies this case considerably.

The following proposition gives a criterion for maps to genus 1 curves.
\begin{proposition}
Suppose that there exists a morphism $f : X \to E$ of degree 3, where $E$ is an elliptic curve, such that the discriminant curve is another elliptic curve. Then either $X$, or a geometrically integral, unramified, double cover, $Y$ has a map $g : Y \to E'$ of degree 3 and Galois group $C_3$ where $E'$ is an elliptic curve.
\label{gen1}
\end{proposition}
\begin{proof}
As the discriminant curve of $f$, $E'$, has genus 1, it must either be $E$, or an unramified double cover of $E$ (i.e.\ a 2-isogenous elliptic curve). In the former case, $X$ satisfies the assumptions as the discriminant of $f$ is a square on $E$ and $f$ has Galois group $C_3$, so we assume the second. Let the fibre product of $X$ and $E'$ over $E$ be $Y$. As $E' \to E$ is unramified and degree 2, the same is true of $Y \to X$. As $X \to E$ is degree 3, $Y \to E'$ is, and moreover, since the discriminant of $X \to E$ is a square in $Y$, this map is cyclic. 
\end{proof}

Combining Propositions~\ref{weier} and ~\ref{gen1} with Theorem~\ref{classgen5} gives \Cref{mainthm_comp}.

We illustrate this proposition with an example. 

\begin{example}
Let $X$ be the genus 5 hyperelliptic curve given by the following. 
$$y^2 = (x - 1)(x + 1)(x^9 - x^7 + x^6 + 2x^5 - 3x^4 - x^3 + 3x^2 - 1)$$
This curve has no non-trivial automorphisms (computed in \texttt{Magma} \cite{mag}). As unramified double covers of $y^2 = f(x)$ are all of the form $z^2 = g(x)$, $w^2 = h(x)$ with $f(x) = g(x)h(x)$, there are only 3 defined over $\mathbb{Q}$. All 3 covers come from a factorisation with either a linear or quadratic factor, and so the covers are hyperelliptic by projecting onto this. The Castelnuovo--Severi inequality shows these cannot triple-cover elliptic covers, since all hyperelliptic curves with a degree 3 maps to elliptic curves have genus at most $5$, but these have genus 9. As $X$ cannot be trigonal, again by the Castelnuovo--Severi inequality, this shows that $X$ has only finitely many cyclic cubic points over $\mathbb{Q}$.

On the other hand, this curve has infinitely many cubic points. It is isomorphic to the curve $y^2 = (x^3 - x)^4 - (x^3 - x)^3 + (x^3 - x)$, which triple covers the genus 1 curve, $y^2 = x^4 - x^3 + x$. This is an elliptic curve, isomorphic to $y^2 = x^3 - x + 1$, which has positive rank. Pulling back the rational points on the elliptic curve gives rise to infinitely many cubic points on $X$.
\end{example}

\section{Curves with genus 2, 3 or 4}
Under additional assumptions, it is possible to make statements about lower genus curves. The statements in this section are also true for curves of larger genus, however, the focus is on genus 3 and 4 as the results of the previous section do not apply in these cases.

\begin{theorem}
Let $X$ be a curve of genus $g \geq 3$ and assume that the Jacobian of $X$ contains no elliptic curves, and if $g = 3$, assume also the Bombieri--Lang conjecture for $X^3 / C_3$. Then $X$ has infinitely many cyclic cubic points if and only if there exists a map $f : X \to \mathbb{P}^1$ of degree 3 and such that $\mathbb{P}^1_{\Delta(f)}$ is either a $\mathbb{P}^1$ or a positive rank elliptic curve.
\end{theorem}
\begin{proof}
We start by showing that the image of $(X^3 / C_3)(K)$ in $W^{(3)}_X(K)$ is finite. 

If $g \geq 4$, then $W^{(3)}_X(K)$ is finite, unless $W^{(3)}_X$ contains an abelian subvariety. By \Cref{abrhar}, this abelian variety must be an elliptic curve, but by assumption the Jacobian of $X$ does not contain an elliptic curve.

If $g = 3$, as $X^3 / C_3$ is of general type \cite{Raw}, the Bombieri--Lang conjecture implies $(X^3 / C_3)(K)$ is contained in a proper subvariety, $Z$, of $X^3 / C_3$. As the Jacobian of $X$ contains no elliptic curves, it is simple. In particular, the image of $Z$ in the Jacobian can contain no abelian subvarieties. 

In both cases, the image of the $K$-points is finite. Therefore, if $(X^3 / C_3)(K)$ is infinite, then there are infinitely many rational points which map to a fibre of $\Sym^3 X \to W^{(3)}_X$. Applying Theorem~\ref{finmap} gives the result in one direction.

The opposite direction is Theorem~\ref{mapinf}.
\end{proof}

Combining this with Proposition~\ref{weier} gives \Cref{mainthm2}. The assumptions can often be verified by Galois theoretic considerations.

\begin{theorem}
Let $X$ be a curve of genus $g \geq 3$. If $g = 3$, assume the Bombieri--Lang conjecture. Suppose that the Jacobian of $X$ contains no elliptic curves and that the Galois action on the Weierstrass points of $X$ is 2-transitive. Then $X$ has finitely many cyclic cubic points. 
\label{gal}
\end{theorem}
\begin{proof}
Let $P_1, ..., P_n \in X(\bar{K})$ be the Weierstrass points, where $n \geq 2g + 2$. If $3P_1 \sim 3P_i$ for some $i$, then by the 2-transitivity of the Galois action, $3P_j \sim 3P_k$ for any $j$ and $k$, in particular, $3P_1 \sim 3P_j$ for all $j$. As the linear system $|3P_1|$ is one-dimensional, this shows that there exists a map $f : X \to \mathbb{P}^1$ of degree 3, such that $3P_j$ is a fibre for all $j$. This map is therefore ramified in at least $2g + 2$ points, and with index 3 at these. The genus of $X$ therefore satisfies $2g - 2 \geq -6 + (2g + 2) \times 2 = 4g - 2$ by Riemann-Hurwitz for $f$. This is not possible, so $3P_1 \not\sim 3P_i$ for all $i$. As this exhausts all Weierstrass points, there are no 2 distinct points with $3P \sim 3Q$, and so by \Cref{weier}, $X$ has only finitely many cyclic cubic points. 
\end{proof}

For hyperelliptic curves, this gives the following corollary.
\begin{corollary}
Let $f \in K[x]$ be a polynomial of even degree $2g + 2$, $g \geq 3$, and assume the Galois group is either $S_{2g + 2}$ or $A_{2g + 2}$. If $g = 3$, assume the Bombieri--Lang conjecture. Then the hyperelliptic curve $y^2 = f(x)$ has only finitely many cyclic cubic points.
\end{corollary}
\begin{proof}
The Weierstrass points of a hyperelliptic curve of this form are exactly the roots of $f$. By assumption, the Galois group is $S_{2g + 2}$ or $A_{2g + 2}$ which act 2-transitively. Additionally, by a theorem of Zarhin \cite{Zar}, the Jacobian of the curve is simple, and so does not contain an elliptic curve. The assumptions of the previous theorem are now satisfied.
\end{proof}

We now give an example of a non-hyperelliptic genus 3 curve where we can verify the assumptions of the theorem.
\begin{corollary}
Assuming the Bombieri--Lang conjecture, the modular curve $X_{ns}^+(13)$ has only finitely many points defined over cyclic cubic extensions of $\mathbb{Q}$.
\label{ns}
\end{corollary}
\begin{proof}
The curve has the following model in $\mathbb{A}^2$, $xy^3 + x^2y^2 + y^3 + 2xy^2 - x^3 + 2xy + 2x - y = 0$. This curve is genus 3 and has simple Jacobian \cite{LMFDB}. For a plane quartic, the Weierstrass points are exactly the flex points and these can be computed as the vanishing of the Hessian determinant of the defining equation. The intersection can be computed in \texttt{Sage} \cite{sag} using Gr\"obner bases, which gives the minimal polynomial for the $y$-coordinates. This polynomial has Galois group $S_{24}$ (computed by \texttt{Magma} \cite{mag}). As $S_{24}$ acts 2-transitively on the 24 flexes, this curve satisfies the assumptions of the theorem.
\end{proof}

Instead of imposing geometric conditions, it can also be useful to, instead, restrict the arithmetic of the curve.

\begin{theorem}
Let $X$ be a curve of genus $g \geq 2$, and suppose that the Jacobian has rank 0, i.e.\ $J_X(K)$ is finite. Then $X$ has infinitely many cyclic cubic points if and only if there exists a map $f : X \to \mathbb{P}^1$ of degree 3 and such that $\mathbb{P}^1_{\Delta(f)}$ is either a $\mathbb{P}^1$ or a positive rank elliptic curve.
\end{theorem}
\begin{proof}
Since $J_X(K)$ is finite, the same is true of $W^{(3)}_X(K)$. If $(X^3 / C_3)(K)$ is infinite, then infinitely many points are mapped to the same point in $W^{(3)}_X(K)$. Theorem~\ref{finmap} gives the forward implication. Theorem~\ref{mapinf} gives the converse.
\end{proof}

Combining this with Proposition~\ref{weier}, gives \Cref{mainthm3}. 

One application of this is Theorem 5 in \cite{Raw}, and we give another application here in genus 2.
\begin{theorem}
There are finitely many elliptic curves defined over cyclic cubic number fields with a 22-isogeny.
\end{theorem}
\begin{proof}
The modular curve $X_0(22)$ has genus 2, and as the level is square-free by a theorem of Ogg \cite{Ogg}, the automorphism group is generated by the Atkin-Lehner involutions, which commute and have order 2. In particular, this precludes any automorphisms of order 3. Using the equations of Dokchitser-Doris \cite{DokDor} with an explicit model of $X_0(22)$ (e.g.\ from \texttt{Magma}'s \texttt{SmallModularCurve} database \cite{mag}), the projective 3-torsion (equivalently, the cyclic 3-isogenies) is found to be defined over 4 quartic fields and a degree 24 number field. Finally, as the cusp forms of level 22 all have analytic rank 0, the Jacobian of $X_0(22)$ has rank 0 over $\mathbb{Q}$. Combining the preceding theorem and \Cref{gen2weier} shows that $X_0(22)$ has only finitely many cyclic cubic points over $\mathbb{Q}$. 
\end{proof}

We close the section by remarking that it is not possible to extend these results, in this form, to any genus 3 curve. 
\begin{example}
Let $f(x)$ be a degree 12 polynomial where $A_4$ acts transitively on the roots via $\mathrm{PGL}_2(\mathbb{C})$. Here, we use $7(8x^3 + 1)(x^3 + 1)(x^2 - 5x + 7)(x^2 + x + 7)(x^2 + 4x + 7)$, where the action of $A_4$ is generated by $\sigma = x \mapsto \omega x$ with $\omega^3 = 1$, and $\tau = x \mapsto \frac{x + 2}{x - 1}$. These extend to automorphisms of the curve $Y : y^2 = f(x)$. Taking the quotient of $Y$ by $\sigma$ gives a map to the elliptic curve $E : y^2 = x^3 - 672x + 6840$ (which has rank 1 over $\mathbb{Q}$). Instead, quotienting by $\tau$ composed with the hyperelliptic involution gives the following genus 3 curve. 
$$X : y^2 = 14(2x - 1)(x + 6)(4x^2 - 6x + 9)(x^2 + 2x + 4)(x^2 + 6x - 3)$$
Pulling back rational points on $E$ to $Y$, shows $Y$ has infinitely many cyclic cubic points (over $\mathbb{Q}(\omega)$). Pushing them forward to $X$ shows the same is true for $X$. We now show there is no degree 3 map from $X$ to $\mathbb{P}^1$ or $E'$ for any $E'$ isogenous to $E$. As the genus of $X$ is 3, and it is hyperelliptic, it has no degree 3 morphisms to $\mathbb{P}^1$ by the Castelnuovo--Severi inequality. The curve $X$ has an involution given by $x \mapsto \frac{-1}{3x}$, and quotienting out by this automorphism gives $E$. If $X$ had both a degree 2 and degree 3 map to $E$, then by considering the graph of these maps in $E \times E$, this would show the Jacobian of $X$ is isogenous to $E^2 \times F$ where $F$ is another elliptic curve. Combining the previously described automorphism of $X$ with the hyperelliptic involution gives a 3rd, and the quotient by this is a genus 2 curve with automorphism group $D_6$. As the automorphism group is non-abelian, the Jacobian of this curve must be geometrically isogenous to a square. This leaves 2 options; either the Jacobian of $X$ is geometrically isogenous to the cube of $E$ or to $E \times F^2$ with $E$ and $F$ not geometrically isogenous. By computing the geometric endomorphism ring in \texttt{Magma} \cite{mag}, it must be the latter. This rules out $X$ having a degree 3 map to $E$ (or an isogenous elliptic curve).
\end{example}

\section{Integral points}
Let $\mathcal{X}$ be an $O_{K, S}$-model for a proper open subset of $X$, where $S$ is a finite set of places of $K$. Define $D$ as the complement of the (generic fibre) of $\mathcal{X}$ in $X$, and since $X$ is a proper subset, we have $\# D \geq 1$.

For integral points, the situation is much simpler than in \Cref{abrhar}, as the following two theorems of Levin, \cite{levin}, classify when a curve has infinitely many integral cubic points, and their origin. The first statement, \cite[Theorem 1.2, Theorem 4.3]{levin}, is akin to \cite[Theorem 1 and Conjecture]{AbrHar}, although the theorem of Levin provides more control.

\begin{theorem}
 Assume $\mathcal{X}$ has infinitely many $\mathcal{O}_{L, S}$-points with $[L : K] \leq d$, for an integer $d$. Then the following hold.
 \begin{enumerate}
  \item After base-changing to $\bar{K}$, there exists a morphism $f : X \to \mathbb{P}^1$ of degree $\leq d$ such that $D \leq f^*(0) + f^*(\infty)$.
  \item If the points of $D$ are elements of $X(K)$, and $\# D > d$, then the map in the previous case is defined over $K$. Moreover, all but finitely many of the $\mathcal{O}_{L, S}$ arise as pre-images of rational points under $K$-rational maps $f : X \to \mathbb{P}^1$ of degree at most $d$ with $D \leq f^*(0) + f^*(\infty)$. 
 \end{enumerate}
 The converse is also true: if there exists $f : X \to \mathbb{P}^1$ of degree $\leq d$ such that $D \leq f^*(0) + f^*(\infty)$, then after enlarging $S$, there exist infinitely many $\mathcal{O}_{L, S}$-points on $\mathcal{X}$ with $[L : K] \leq d$.
 \label{lev_rat}
\end{theorem}

It is also possible to reach similar conclusions with much weaker assumptions on $D$, although the conclusions are weaker. For example, we state here \cite[Theorem 5.5]{levin}.
\begin{theorem}
 All but finitely many $\mathcal{O}_{L, S}$-points on $\mathcal{X}$, as $L$ varies over the fields with $[L : K] \leq d$ arise from fibres of maps $f : X \to \mathbb{P}^1$, with $\deg f \leq d$.
 \label{lev_gen}
\end{theorem}

We now apply this result to cyclic cubic points. 
\begin{theorem}
Let the genus of $X$ be at least 3, and assume $\mathcal{X}$ has infinitely many integral cyclic cubic points, then there exists a map $f : X \to \mathbb{P}^1$ of degree 3 such that $\#f(D) \leq 2$, and the discriminant curve $\mathbb{P}^1_{\Delta(f)}$ is a $\mathbb{P}^1$.
\label{int_geq_3}
\end{theorem}
\begin{proof}
We first consider the image of $(\mathcal{X}^3 / C_3)(\mathcal{O}_{K, S})$ in $\Sym^3 \mathcal{X}(\mathcal{O}_{K, S})$. By \Cref{lev_gen}, all but finitely many elements of $\Sym^3 \mathcal{X}(\mathcal{O}_{K, S})$ are contained in $\mathbb{P}^1$s inside $\Sym^3 X$. As in the case of rational points, the morphisms from $X$ corresponding to these $\mathbb{P}^1$s must have degree exactly 3, otherwise the fibres will not actually be defined over a degree 3 field. For curves of genus at least 3, there are finitely many such morphisms over $K$. Each $\mathbb{P}^1$ will contain finitely many $\mathcal{O}_{K, S}$-points unless it is punctured in less than 3 places (by Siegel's Theorem), and the number of punctures is the size of $f(D)$, where $f$ is the corresponding morphism from $X$. 

The $\mathcal{O}_{K, S}$-points of $\mathcal{X}^3 / C_3$ must therefore be contained (with finitely many exceptions) in the pre-image of these rational curves in $\Sym^3 \mathcal{X}$. These pre-images are the discriminant curves of the corresponding morphisms (unless the morphism is Galois, when it is two copies of the discriminant curve, which is $\mathbb{P}^1$). For one of these discriminant curves to have infinitely many integral points, it must be a $\mathbb{P}^1$ as it has at least one puncture.
\end{proof}

Proposition~\ref{weier} can be modified to incorporate this stricter genus bound on the discriminant, to show there are $g + 1$ points satisfying the linear equivalence condition.

\begin{theorem}
Now assume $X$ has genus 2, $\mathcal{X}$ has infinitely many integral cyclic cubic points, and $D$ contains at least 4 $K$-rational points. Then there exists a map $f : X \to \mathbb{P}^1$ of degree 3 such that $\# f(D) \leq 2$ and the discriminant curve is a $\mathbb{P}^1$.
\label{int_2}
\end{theorem}
\begin{proof}
Let $\mathcal{X}'$ be an integral model for $X \setminus D'$, where $D'$ consists of the 4 rational points in the support of $D$, such that $\mathcal{X}$ is an open subset of $\mathcal{X}$. By \Cref{lev_rat}, the integral cubic points on $\mathcal{X}'$, with finitely many exceptions, come from morphisms $f : X \to \mathbb{P}^1$, with $f(D') \subset \{0, \infty\}$ of degree at most 3. As in the case of rational points, these must have degree exactly 3.  

We claim there are finitely many such $f$. Let the 4 rational points be $P_1, P_2, P_3$ and $P_4$. After possibly replacing $f$ by $\frac{1}{f}$ and relabelling the $P_i$, we may assume $f^*(\infty) = P_1 + P_2 + P_3$ and $f(P_4) = 0$, or $f^*(\infty) = P_1 + P_2 + Q_1$ and $f^*(0) = P_3 + P_3 + Q_2$ for some $Q_i \in X(\bar{K})$. In the first case, $f$ is uniquely determined up to scaling, since a degree 3 divisor has at most 2 sections, and the requirement that $f(P_4) = 0$ imposes one condition. In the second, we have to show there are finitely many choices of the $Q_i$ such that $P_1 + P_2 + Q_1 \sim P_3 + P_4 + Q_2$. These two points, $Q_i$, satisfy $Q_1 - Q_2 \sim (P_3 + P_4) - (P_1 + P_2)$. If there are infinitely many choices of $Q_i$ such that this holds, we can make two different choices of such $Q_i$, say $Q_i$ and $Q_i'$, such that $Q_1' + Q_2$ does not move (equivalently, $Q_1'$ is not the hyperelliptic conjugate of $Q_2$). By assumption, $Q_1' - Q_2' \sim Q_1 - Q_2$, and so, $Q_1' + Q_2 \sim Q_1 + Q_2'$. Since $Q_1' + Q_2$ does not move, this linear equivalence is an equality, and so $Q_1 = Q_1'$ or $Q_2$. The first possibility forces $Q_2 = Q_2'$ and so they are not two different choices. For the second, this forces $P_1 + P_2 \sim P_3 + P_4$, in particular $P_1 + P_2$ moves. However, the divisor $P_1 + P_2 + Q_1$ then has a basepoint, and so there is no $f$ of degree 3 with the desired zeroes and poles.

We now restrict to integral cubic points on $\mathcal{X}$. These are a subset of those of $\mathcal{X}'$, and so also come from these maps. Moreover, for their to be infinitely many integral points on the corresponding $\mathbb{P}^1$s inside $\Sym^3 \mathcal{X}$, they must be punctured in no more than 2 places, again by Siegel's Theorem. This forces $f(D)$ to also be two points.

For infinitely many of these integral points on $\Sym^3 \mathcal{X}$ to pull back to integral points on $\mathcal{X}^3/C_3$, the discriminant curve corresponding to the morphism must be a $\mathbb{P}^1$.
\end{proof}

To avoid checking the discriminant of every degree 3 map, we give a simpler sufficient condition for finiteness of integral cyclic cubic points.

\begin{theorem}
Let $X$ be a curve of genus $g \geq 2$. If $g = 2$, assume that $D$ contains 4 $K$-rational points. If $\#D \geq 5$ and $X$ has no automorphism of order 3, then $X$ has finitely many integral cyclic cubic points.
\end{theorem}
\begin{proof}
By Theorems \ref{int_geq_3}, \ref{int_2}, if $X$ has infinitely many integral cyclic cubic points, they come from morphisms $f : X \to \mathbb{P}^1$ of degree 3 and with rational discriminant curve. As $X$ has no automorphism of order 3, the discriminant curve must be a double cover of the base $\mathbb{P}^1$ ramified at 2 points. It remains to show that removing the points of the discriminant curve (inside $X^3 / C_3$) with support in $D$ punctures it in 3 or more places, as there are only finitely many integral points on such curves by Siegel's Theorem. 

If a point of $D$ does not lie above a branch point of $f$, then removing the corresponding fibre removes 2 points of the discriminant curve. There are at least 2 more points in $D$ that are not in this fibre. Removing the fibres containing them removes at least 1 more point.

If every point of $D$ lies above a branch point of $f$, then each fibre of $f$ which meets $D$ consists of at most 2 points. This shows $f(D)$ is at least 3 points, and so removing the fibres punctures the discriminant curve in at least 3 places.  
\end{proof}

As in Theorem~\ref{gal}, we can use Galois-theoretic assumptions to quickly check the hypotheses of Proposition~\ref{weier}. 

\begin{theorem}
Let $X$ be a curve of genus $g \geq 3$ such that the Galois action on the Weierstrass points of $X$ is 2-transitive. Then for any integral model, $\mathcal{X}$, of a proper open subset of $X$, $\mathcal{X}$ has only finitely many integral cyclic cubic points.
\end{theorem}
\begin{proof}
If $X$ has infinitely many integral cyclic cubic points, then there exists a map $f : X \to \mathbb{P}^1$ with discriminant curve $\mathbb{P}^1_{\Delta(f)}$, which is a $\mathbb{P}^1$. By Proposition~\ref{weier}, $X$ has $g + 1$ points $P_1, ..., P_{g + 1} \in X(\bar{K})$ with $3P_i \sim 3P_j$. However, by the same argument as Theorem~\ref{gal}, there are $2g + 2$ such linearly equivalent points, which would contradict the Riemann-Hurwitz formula.
\end{proof}

We close with an example.
\begin{theorem}
For any set of primes, $S$, there exist only finitely many elliptic curves defined over cyclic cubic number fields with a 43-isogeny and potential good reduction at every prime, except those lying above $S$.
\end{theorem}
\begin{proof}
The elliptic curves with this property are those with $S$-integral $j$-invariant. These are parameterised by the open modular curve $Y_0(43)$. The compactification, $X_0(43)$, is obtained by adding two points and is a plane quartic curve. 
The Weierstrass points of $X_0(43)$ can be computed from an explicit model using the Hessian determinant. For this curve, there are 24 Weierstrass points, which come in 12 pairs. The action of Galois is not 2-transitive, but there are only 2 orbits when acting on pairs, the 12 marked pairs, and the rest. We now argue as in the previous theorem. Since any $f : X_0(43) \to \mathbb{P}^1$ with rational discriminant curve has at least $3 + 1$ triple ramification points (which are also Weierstrass points), $P_1, ..., P_4$, at most one of $P_2, P_3$ and $P_4$ is the marked pair of $P_1$. After relabelling if necessary, $P_3$ is not the pair to $P_1$. Since $3P_1 \sim 3P_3$, by the transitivity of the Galois action on the unmarked pairs, the same is true for any unmarked pair. In particular, $3P_1$ is linearly equivalent to $3Q$ for 23 different $Q \in X_0(43)(\bar{\mathbb{Q}})$. Since $|3P_1|$ is one-dimensional, the $3Q$ are all geometrically fibres of $f$, but $f$ has at most five triple ramification points (by Riemann-Hurwitz). Therefore, $X_0(43)$ does not have any morphisms to $\mathbb{P}^1$ with rational discriminant curve, and so $Y_0(43)$ only has finitely many cyclic cubic $S$-integral points.
\end{proof}

\printbibliography

\end{document}